\documentclass[fleqn,reqno,11pt,a4paper,final]{amsart}

\usepackage[a4paper,left=30mm,right=30mm,top=30mm,bottom=30mm,marginpar=20mm]{geometry}
\usepackage{amsmath}
\usepackage{amssymb}
\usepackage{amsthm}
\usepackage{amscd}
\usepackage[ansinew]{inputenc}
\usepackage{cite}
\usepackage{bbm}
\usepackage{color}
\usepackage[english=american]{csquotes}
\usepackage[final]{graphicx}
\usepackage{hyperref}
\usepackage{calc}
\usepackage{mathptmx}
\usepackage{mathtools}
\usepackage{perpage} %the perpage package
\MakePerPage{footnote}
\usepackage{t1enc}
\graphicspath{{../Pictures/}}

\numberwithin{equation}{section}

\newtheoremstyle{thmlemcorr}{10pt}{10pt}{\itshape}{}{\bfseries}{.}{10pt}{{\thmname{#1}\thmnumber{ #2}\thmnote{ (#3)}}}
\newtheoremstyle{thmlemcorr*}{10pt}{10pt}{\itshape}{}{\bfseries}{.}\newline{{\thmname{#1}\thmnumber{ #2}\thmnote{ (#3)}}}
\newtheoremstyle{remexample}{10pt}{10pt}{}{}{\bfseries}{.}{10pt}{{\thmname{#1}\thmnumber{ #2}\thmnote{ (#3)}}}
\newtheoremstyle{ass}{10pt}{10pt}{}{}{\bfseries}{.}{10pt}{{\thmname{#1}\thmnumber{ A#2}\thmnote{ (#3)}}}

\theoremstyle{thmlemcorr}
\newtheorem{theorem}{Theorem}
\numberwithin{theorem}{section}
\newtheorem{lemma}[theorem]{Lemma}

\newtheorem{definition}[theorem]{Definition}

\theoremstyle{thmlemcorr*}
\newtheorem{theorem*}{Theorem}
\newtheorem{lemma*}[theorem]{Lemma}
\newtheorem{corollary*}[theorem]{Corollary}
\newtheorem{proposition*}[theorem]{Proposition}
\newtheorem{problem*}[theorem]{Problem}
\newtheorem{conjecture*}[theorem]{Conjecture}
\newtheorem{definition*}[theorem]{Definition}

\theoremstyle{remexample}
\newtheorem{remark}[theorem]{Remark}

\theoremstyle{ass}

\newcommand{\Rbb}{\mathbb{R}}

\newcommand{\Tbb}{\mathbb{T}}
\newcommand{\T}{\mathbb{T}^d}

\newcommand{\bFormula}[1]{ \begin{equation} \label{#1} }
\newcommand{\eF}{ \end{equation} }

\newcommand{\Ov}[1]{\overline{#1}}

\newcommand{\vr}{\varrho}

\newcommand{\vc}[1]{{\bf #1}}

\newcommand{\Div}{{\rm div}_x}
\newcommand{\Grad}{\nabla_x}

\newcommand{\dx}{{\rm d} {x}}
\newcommand{\dt}{{\rm d} t }

\newcommand{\dxdt}{\dx \ \dt}
\newcommand{\intO}[1]{\int_{\Omega} #1 \ \dx}

\newcommand{\n}{\tilde{n}}

\DeclareMathOperator{\id}{id}

\DeclareMathOperator{\diverg}{div}

\DeclareMathOperator{\supp}{supp}
\DeclareMathOperator{\prob}{Prob}
\DeclareMathOperator{\sgn}{sgn}
\DeclareMathOperator{\mat}{Mat}
\DeclareMathOperator{\cof}{cof}

\newcommand{\norm}[1]{\|#1\|}

\newcommand{\R}{\mathbb{R}}

\newcommand{\Z}{\mathbb{Z}}

\newcommand{\eps}{\epsilon}

\newcommand{\nutt}{\nu_{(x,\tau)}}
\newcommand{\nut}{\nu_{(x,t)}}

\newcommand{\sut}{\sigma_{(x,t)}}
\newcommand{\intd}{\int_{T^d}}
\newcommand{\tplus}{\T\times\R_+}
\newcommand{\dd}{\mathrm{d}x\mathrm{d}t}

\def\XXint#1#2#3{{\setbox0=\hbox{$#1{#2#3}{\int}$}
\vcenter{\hbox{$#2#3$}}\kern-.5\wd0}}

\renewcommand{\eps}{\varepsilon}
\renewcommand{\epsilon}{\varepsilon}
\renewcommand{\phi}{\varphi}
\begin{document}
\title{Relative entropy method for measure-valued solutions in natural sciences}
\author{Tomasz D\k{e}biec \and Piotr Gwiazda \and Kamila \L{}yczek \and Agnieszka \'{S}wierczewska-Gwiazda.
}
\address{\textit{Tomasz D\k{e}biec:} Institute of Applied Mathematics and Mechanics, University of Warsaw, Banacha 2, 02-097 Warszawa, Poland}
\email{t.debiec@mimuw.edu.pl}
\address{\textit{Piotr Gwiazda:} Institute of Mathematics, Polish Academy of Sciences, \'Sniadeckich 8, 00-656 Warszawa, Poland, and Institute of Applied Mathematics and Mechanics, University of Warsaw, Banacha 2, 02-097 Warszawa, Poland}
\email{pgwiazda@mimuw.edu.pl}
\address{\textit{Kamila \L{}yczek:} Institute of Applied Mathematics and Mechanics, University of Warsaw, Banacha 2, 02-097 Warszawa, Poland}
\email{k.lyczek@mimuw.edu.pl}
\address{\textit{Agnieszka \'{S}wierczewska-Gwiazda:} Institute of Applied Mathematics and Mechanics, University of Warsaw, Banacha 2, 02-097 Warszawa, Poland}
\email{aswiercz@mimuw.edu.pl}

\maketitle

\begin{abstract}
We describe the applications of the relative entropy framework introduced in \cite{Daf}. In particular uniqueness of an entropy solution is proven for a scalar conservation law, using the notion of measure-valued entropy solutions. Further we survey recent results concerning measure-valued-strong uniqueness for a number of physical systems -- incompressible and compressible Euler equations, compressible Navier-Stokes, polyconvex elastodynamics and general hyperbolic conservation laws, as well as long-time asymptotics of the McKendrick-Von Foerster equation.
\end{abstract}

\section{Introduction}
The origins of the relative entropy method can be traced back to physics. The underlying principle behind it is the simple idea to measure in a certain way how much two evolutions of a given physical system, whose initial states are "close", differ and to investigate how this "distance" evolves in time. 
This framework, closely related to the second law of thermodynamics, is a useful tool in obtaining a variety of interesting analytical results. For instance it can be used to show uniqueness of solutions to a conservation law in the scalar case, while for many systems of equations it provides the so-called weak-strong uniqueness property, i.e. establishes uniqueness of classical solutions in a wider class of weak solutions. This application, first described by Dafermos \cite{Daf,Daf2}, will be highlighted in this article.\\
\indent Other areas where relative entropy method is found useful include stability studies, asymptotic limits and dimension reduction problems (e.g. \cite{CleoTz}, \cite{GiTz}, \cite{FeJiNo}, \cite{BeFeNo}). The method is also applied to problems arising from biology, cf. \cite{MiMiPe}, \cite{MiMiPe2}, \cite{Pert}, \cite{GwiWie}, known in this context as General Relative Entropy (GRE). It is essentially used for showing asymptotic convergence of solutions to steady-state solutions.\\
\indent On the level of weak solutions for various physical systems, including Navier-Stokes and Euler, the story seems quite complete. However recent years have delivered many new results on the level of measure-valued solutions (e.g. \cite{BrzFei}, \cite{DeStTz}, \cite{GwSwWi}, \cite{FeGSG}). This shows that even though mv solutions are considered a very weak notion of solution, not carrying much information, they do play an important role in the analysis of physical systems.
We begin our discussion on the level of the scalar conservation law
\begin{equation}\label{conservLawIntro}
\begin{split} 
\partial_t u(x,t) + \diverg_x f(u(x,t)) &= 0,\\
u(x, 0)&=u_0(x),
\end{split}
\begin{split}
&\;\; \text{in}\; \tplus,\\
&\;\; \text{in}\; \T.
\end{split}
\end{equation}
in the framework of measure-valued solutions. Here $\R_+ = [0,+\infty)$, $\T = (\R/2\pi\Z)^d$ and $u_0$ is a given initial datum. The main ideas come from Tartar \cite{Tar} and DiPerna \cite{DiP}, who defined entropy mv solutions in the language of classical Young measures, see also \cite{rokyta}.
\begin{definition}[Measure-valued solution]
A measurable measure-valued map
$$\nu: (x,t) \to \nu_{(x,t)} \in \prob(\R)$$
from $\T\times\R_+$ to the space of probability measures on $\R$ is a \emph{measure-valued solution} of ~\eqref{conservLawIntro} if
\begin{equation}
\partial_t \langle \nut, \lambda \rangle + \diverg_x \langle \nut, f(\lambda)\rangle = 0
\end{equation}
in the sense of distributions, that is 
\[\int_{\R_+} \int_{\T} \left\{\langle \nu_{(x,t)}, \lambda \rangle\partial_t\phi +\langle \nu_{(x,t)}, f(\lambda)\rangle \nabla_x\phi\right\}\ \dxdt = 0\]
for all $\phi \in \mathrm{C}_c^1(\T\times\R_+)$.
\end{definition}
Here $\langle\cdot,\cdot\rangle$ denotes the duality pairing between $\mathrm{C}_0(\R)$, the closure with respect to the supremum norm of the space of continuous functions on $\R$ with compact support, and $\mathcal{M}(\R)$, the space of signed Radon measures on $\R$, i.e.
\[\langle\mu,g(\lambda)\rangle\coloneqq\int_\R g(\lambda)\ \mathrm{d}\mu(\lambda).
\]
By measurability of a measure $\nu$ we mean the weak$^*$-measurability of the measures $\nu_{(x,t)}$, i.e. measurability of the map
\[ 
(x,t)\mapsto\left\langle\nu_{(x,t)}, g(\lambda)\right\rangle
\]
for each $g\in \mathrm{C}_0(\R)$.
Such a measure-valued solution often arises from a weakly convergent approximating sequence. The framework of Young measures is in a sense a way of immersing the initial problem into a wider space -- in this way one gains linearity at the cost of having to deal with measure spaces rather than function spaces. In other words Young measures allow to deal with the non-commutativity of weak limits with nonlinearities. Indeed, it can be shown that if $\{u_k\}$ is a sequence uniformly bounded  in $L^p$, then, along a non-relabelled subsequence, the weak limit of $\{f(u_k)\}$ can be represented by means of a parameterised family of measures, a Young measure.
\begin{lemma}[Fundamental Lemma of Classical Young Measures]
Let $\{u_k\}:\tplus\to\R$ be a sequence uniformly bounded in $L^\infty(\R_+; L^p(\T))$, i.e.
\begin{equation}\label{normboundfund}
\sup\limits_{k}\norm{u_k(\cdot, t)}_{L^p(\T)}\leq C,\hspace{0.25cm}\text{for almost every}\;t\in\R_+.
\end{equation}
Then there exists a subsequence, still denoted $\{u_k\}$ and a measurable measure-valued mapping $\nu:\tplus\to\prob(\R)$, such that for each $g\in \mathrm{C}(\R^d)$, satisfying the growth condition
\[|g(\lambda)|\leq C(1+|\lambda|^q)\hspace{0.25cm}\text{for}\;\;1\leq q < p,
\] 
the weak limit of $g(u_k(x,t))$ exists and is represented by $\left\langle\nut,g(\lambda)\right\rangle$, i.e.
\begin{equation}\label{weakstarlim}
\lim\limits_{k\to\infty}\;\int_{\tplus}g(u_k(x,t))\phi(x,t)\ \dxdt = \int_{\tplus}\langle\nut,g(\lambda)\rangle\phi(x,t)\ \dxdt,\hspace{0.25cm}\forall\phi\in L^\infty(\tplus).
\end{equation}
\end{lemma}
DiPerna \cite{DiP} and later Szepessy \cite{Szepessy} under different assumptions on the continuous flux showed the so-called averaged contraction principle, which is a crucial estimate in showing uniqueness of mv solutions, and is essentially a form of relative entropy inequality. Importantly the uniqueness result is proven under the assumption that the initial data is a Dirac delta measure. For non-atomic initial data uniqueness might fail in the class of measure-valued solutions, even in the scalar case -- and even under an entropy inequality, as it provides information only on certain moments of the solution, cf. \cite{FjKaMiTa}\\
\indent To show existence of a measure-valued entropy solution to ~\eqref{conservLawIntro} a parabolic approximate problem is considered. This generates a sequence of approximate solutions, which can be shown to be uniformly integrable. Thus one can see that there is going to be no concentration effect. However, other approximation schemes can be considered, which will not posses sufficient integrability. 
In fact we prove uniqueness in a wider class of mv solutions not necessarily corresponding to any approximation scheme. 

\indent Further, the situation differs substantially in the case of hyperbolic systems. The result introducing measure-valued solutions for the incompressible Euler describes both oscillations and concentrations, cf \cite{DiMa}, \cite{AlBo}. This is because, contrary to the scalar case, for systems there is usually only one entropy-entropy flux pair forming a companion law. The corresponding entropy inequality lacks the symmetry which is present in the scalar case.\\
\indent One then considers the so-called generalized Young measure. Here by a generalized Young measure we mean a triple $(\nu,m,\nu^\infty)$ describing oscillations, concentrations and concentration-angle respectively. The following representation result, proven in \cite{AlBo} is then true 
\begin{lemma}\label{measureschar}
Let $\{u_k\}:\tplus\to\R^n$ be a sequence bounded in $L^\infty(\R_+; L^1(\T))$, i.e.
\begin{equation}\label{normboundAB}
\sup\limits_{k}\;\norm{u_k(\cdot, t)}_{L^1(\T)}\ \leq C,\hspace{0.25cm}\text{for almost every}\;t\in\R_+.
\end{equation}
Then there exists a subsequence, still denoted $\{u_k\}$, a measurable measure-valued mapping $\nu: \tplus\to\prob(\R)$, a non-negative measure $m$ on $\tplus$ and a parameterised probability measure $\nu^\infty\in L_w^\infty(\tplus,m,\prob(\mathbb{S}^{n-1}))$ such that for each Caratheodory function $g:\tplus\times\R^n\to\R^n$
\begin{equation}\label{weakstarlimAB}
g(x,t,u_k(x,t))\xrightharpoonup{*}\langle\nu_{(x,t)},g(x,t,\lambda)\rangle+\langle\nut^\infty,g^\infty\rangle m
\end{equation}
weak$^{*}$ in the sense of measures, provided the function $g^\infty$, defined to be
\[g^\infty(x,t,\beta)\coloneqq\lim\limits_{s\to\infty}\;\;\lim\limits_{(x',t',\beta')\to(x,t,\beta)}\frac{g(x',t',s\beta')}{s},\hspace{0.25cm}\text{where}\;\beta\in\mathbb{S}^{n-1}.
\]
is of class $\mathrm{C}(\tplus\times\mathbb{S}^{n-1})$.
\end{lemma}
In this article we want to show in as technically simple a way as possible the main idea of this framework. To this end we provide in Section \ref{sec2} the full proof of an averaged contraction principle for the scalar conservation law under the artificial assumption that concentration effects cannot be excluded. This highlights an interesting phenomenon that the concentration measure indeed vanishes thus concluding uniqueness on the basis of averaged contraction principle. The same observation transfers to systems, of course on the level of proving only weak(measure-valued)-strong uniqueness, cf. \cite{BrDLSz}, \cite{GwSwWi}. Interestingly, even more can be claimed, namely that the information about concentration angle does not give any essential information, cf. \cite{FeGSG}. Indeed it is enough to know that the concentration measure appearing in the weak formulation is dominated by the concentration measure coming from the energy/entropy inequality.\\
\indent Then in Sections \ref{sec3}-\ref{sec5} we survey the weak-strong uniqueness results obtained using the relative entropy method for equations of fluid dynamics (Section \ref{sec3}), polyconvex elastodynamics (Section \ref{sec4}) and general systems of conservation laws. Finally we display how the relative entropy method is used for the renewal equation of mathematical biology (Section \ref{sec6}).

%%%%%%%%%%%%%%%%%%%%%%%%%%%%%%%%%%%%%%%%%%%%%%%%%%%%%%%%%%%

\section{Measure-valued Entropy Solutions for Scalar Conservation Laws}\label{sec2}
We consider the Cauchy problem for the scalar conservation law 
\begin{equation}\label{conservLaw}
\begin{aligned} 
\partial_t u + \diverg_x f(u) &= 0,\\
u(x, 0)&=u_0(x),
\end{aligned}
\end{equation}
where $f:\R\to\R^d$ represents the flux of the quantity $u:\tplus\to\R$. We choose to work in the spatially periodic setting to avoid inessential technical issues.\\  
We introduce now the concept of entropy measure-valued solutions, which will then be used to show existence of a unique entropy solution to the problem ~\eqref{conservLaw}. The following will be standing assumptions on the flux function $f$ and the initial datum $u_0$:
%\begin{equation}
\begin{align}
&f\in \mathrm{C}(\R)\hspace{0.25cm}\text{and}\hspace{0.25cm} u_0\in L^1(\T),\\
&|f(\lambda)|\leq C(1+|\lambda|),\quad \text{for some}\; C>0. \label{sublinear}
\end{align}
%\end{equation}
%\begin{equation}\label{sublinear}
%|f(\lambda)|\leq C(1+|\lambda|),\quad \text{for some}\; C>0,
%\end{equation}
%\[\limsup_{|\lambda|\to 0}\;\frac{|f(\lambda)-f(0)|}{|\lambda|^\alpha}<\infty,\hspace{0.25cm}\text{for some}\hspace{0.25cm}\alpha\in (0,1].\]
The definition of a measure-valued entropy solution to ~\eqref{conservLaw} consists of a classical Young measure $\nu$ as well as two concentration measures $m_1$ and $m_2$. We will assume that
\begin{equation}\label{mmbound}
|m_2|(\T\times A)\leq C\;m_1(\T\times A)
\end{equation}
for any borel set $A\subset\R_+$. We will later see that for a solution arising as a limit of an approximating sequence this is guaranteed by characterization of the corresponding concentration measures.
\begin{definition}[Measure-valued entropy solution]
The triple $(\nu,m_1,m_2)$ generated by a sequence which satisfies ~\eqref{normboundAB} is called a \emph{measure-valued entropy solution with concentration} of conservation law ~\eqref{conservLaw} if
\begin{equation}\label{entropycond}
\partial_t \left(\left\langle \nu, |\lambda - k| \right\rangle + m_1\right) + \diverg_x \left(\left\langle \nu, \sgn(\lambda - k)(f(\lambda) - f(k)) \right\rangle + m_2\right) \leq 0
\end{equation}
in the sense of distributions on $\T\times\R_+$ for all $k \in \R$, and if
\begin{equation}\label{mildcts}
\lim_{T \to 0^+}\left\{ \frac{1}{T} \int_0^T \int_{\T} \langle \nu_{(x,t)}, |\lambda - u_0(x)|\rangle\ \dxdt + \frac{1}{T}\int_{\T\times [0,T]}m_1(\dd)\right\}=0.
\end{equation} 
\end{definition}
%Observe that condition ~\eqref{entropycond} is a measure version of the classical Lax entropy inequality for an $L^{\infty}$ solution of ~\eqref{conservLaw}
%\[\partial_t\eta(u) + \diverg_x(q(u))\leq 0\]
%with the entropy pair
%\[\eta(\lambda, k) = |\lambda - k|,\qquad q(\lambda,k) = |f(\lambda) - f(k)|.\]
%It was an observation by Kru\v{z}kov \cite{Kru} that these conditions are equivalent.
Notice that the entropy inequality ~\eqref{entropycond} is sufficient to guarantee that the triple $(\nu, m_1, m_2)$ satisfies also the weak formulation. To see this assume, for simplicity, that the measure $\nu$ has a bounded support $\supp\nut\in(-A,A)$. Then taking $k<-A$ we have $\lambda - k>0$ and ~\eqref{entropycond} becomes for any $\phi\geq 0$
\begin{equation*}
\begin{aligned}
0&\leq\int_{\tplus}\langle\nut,\lambda-k\rangle\partial_t\phi\ \dxdt + \int_{\tplus}\partial_t\phi\ m_1(\dd)\\
&\hspace{0.25cm}+ \int_{\tplus}\langle\nut,f(\lambda)-f(k)\rangle\nabla\phi\ \dxdt + \int_{\tplus}\nabla\phi\ m_2(\dd)\\
&=\int_{\tplus}\langle\nut,\lambda\rangle\partial_t\phi\ \dxdt + \int_{\tplus}\partial_t\phi\ m_1(\dd)\\
&\hspace{0.25cm}+ \int_{\tplus}\langle\nut,f(\lambda)\rangle\nabla\phi\ \dxdt + \int_{\tplus}\nabla\phi\ m_2(\dd).
\end{aligned}
\end{equation*}
Taking in turn $k>A$ one obtains the reversed inequality. It follows that
\[\partial_t(\langle\nut,\lambda\rangle + m_1) + \diverg_x(\langle\nut,f(\lambda)\rangle + m_2) = 0.
\]
Now we prove a uniqueness result concerning mv entropy solutions.
\begin{theorem}
Suppose $(\nu,m_1,m_2)$ and $(\sigma,\overline{m}_1, \overline{m}_2)$ are Young measure-concentration measure triples satisfying ~\eqref{entropycond} and $\nu_{(x,0)} = \sigma_{(x,0)} = \delta_{\{u_0(x)\}}$. Then there exists a function
\begin{equation}\label{regularity}
w\in L^\infty(\R_+;L^1(\T))
\end{equation}
such that
\begin{equation}
\nu_{(x,t)} = \sigma_{(x,t)} = \delta_{\{w(x,t)\}}
\end{equation}
for almost each $(x,t)\in\tplus$ and $m_i = \overline{m}_i = 0$, $i=1,2$.
\end{theorem} 
\begin{proof}
We begin by mollifying the measures $\nu$, $\sigma$, $m_i$ and $\overline{m}_i$, $i=1,2$. Let $\eta$ be a non-negative, symmetric smooth function on $\tplus$ compactly supported in the open unit ball in $\R^d\times\R_+$
%set $\left\{(x,t)\in\R^d\times\R_+\;|\;|(x,t)|<1 \right\}$
and such that $\int_{\R^d\times\R_+}\eta(x,t) = 1$, and let $\eta^\eps(x,t) = \eps^{-(d+1)}\eta(\frac{x}{\eps},\frac{t}{\eps})$. By $\nu^\eps$ we denote the parameterised measure satisfying for all $g\in \mathrm{C}(\R)$
\[\langle \nut^\eps,g\rangle = \int_{\tplus}\eta^\eps(x-x',t-t')\langle\nu_{(x',t')},g\rangle\ \mathrm{d}x'\mathrm{d}t'\;\in \mathrm{C}^\infty(K; \mathcal{M}(\R))
\]
where $K\subset\tplus$ is a set whose $\eps$-neighbourhood is entirely contained in $\tplus$. Furthermore we observe that by virtue of the Riesz representation theorem, for each $(x,t)\in\tplus$ there are bounded measures $\partial_x\nut^\eps$ and $\partial_t\nut^\eps$ s.t. for all $g\in \mathrm{C}(\R)$
\[
\langle\partial_\alpha\nut^\eps,g\rangle = \partial_\alpha\langle\nut^\eps,g\rangle,\quad \text{for}\; \alpha\in\{t,x\},
\]
and the map $(x,t)\mapsto\langle\partial_\alpha\nut^\eps,g\rangle$ is continuous. Observe that mollifying measures $m_i$ and $\overline{m}_i$ with a regular kernel gives smooth functions $m_1^\eps,\overline{m}_1^\eps\in\mathrm{C}^\infty(K;\R_+)$ and $m_2^\eps,\overline{m}_2^\eps\in\mathrm{C}^\infty(K;\R)$. Observe also that $\nu$ and $\sigma$ have finite first moments.\\
We will now show that the regularized measures satisfy the entropy inequality. Let $V\subset\tplus$ be an open bounded set. For sufficiently small $\eps>0$ we have, for each non-negative $\phi\in \mathrm{C}_c^\infty(V)$ and $g\in \mathrm{C}(\R)$,
\begin{equation*}
\begin{aligned}
    \int_{\tplus}\left(\langle\nut,g\rangle\right)\partial_t(\phi*\eta^\eps)\ \dxdt + \int_{\tplus}\partial_t(\phi*\eta^\eps)\ m_1(\dd) \\
     =-\int_{\tplus}\left(\langle\partial_t\nut^\eps,g\rangle + \partial_t m_1^\eps\right)\phi(x,t)\ \dxdt,
\end{aligned}
\end{equation*}
and analogously for the spatial derivative. Hence, choosing in turn $g(\lambda) = |\lambda-\mu|$ and $g(\lambda) = \sgn(\lambda-\mu)(f(\lambda)-f(\mu))$ and using ~\eqref{entropycond}
\begin{equation}
\begin{aligned}
\int_{\tplus}&\left(\left\langle\partial_t\nut^\eps,|\lambda-\mu|\right\rangle+\partial_t m_1^\eps + \left\langle\partial_x\nut^\eps,\sgn(\lambda-\mu)(f(\lambda)-f(\mu))\right\rangle +\partial_x m_2^\eps \right)\phi(x,t)\ \dxdt\\
&\hspace{-0.5cm}=-\int_{\tplus}\left[\langle\nut,|\lambda-\mu|\rangle\partial_t(\phi*\eta^\eps) + \langle\nut,q(\lambda,\mu)\rangle\nabla(\phi*\eta^\eps)\right]\ \dxdt \\
&-\int_{\tplus}\partial_t(\phi*\eta^\eps)\ m_1(\dd) - \int_{\tplus}\nabla(\phi*\eta^\eps)\ m_2(\dd)\\
&\hspace{-0.5cm}\leq 0,
\end{aligned}
\end{equation}
where we denote $q(\lambda, \mu)\coloneqq\sgn(\lambda-\mu)(f(\lambda)-f(\mu))$ for brevity.
Therefore
\begin{equation}\label{smoothentropynu}
\langle\partial_t\nut^\eps,|\lambda-\mu|\rangle+\partial_t m_1^{\eps} + \langle\partial_x\nut^\eps, \sgn(\lambda-\mu)(f(\lambda)-f(\mu))\rangle+\partial_x m_2^{\eps} \leq0
\end{equation}
for all $\mu\in\R$ and $(x,t)\in V$. Symmetrically it can be seen that the triple $(\sigma^\eps, \overline{m}_1^\eps, \overline{m}_2^\eps)$ satisfies ~\eqref{entropycond} for all $\lambda\in\R$. Next we observe that the function $q$ is continuous on $\R^2$ and it has sublinear growth due to growth conditions ~\eqref{sublinear} on $f$. It follows that the maps $\mu\mapsto\int q(\lambda,\mu)\ \mathrm{d}\nut^\eps(\lambda)$ and $\lambda\mapsto\int q(\lambda,\mu)\ \mathrm{d}\sut^\eps(\mu)$ are continuous (by virtue of the dominated convergence theorem). Let now $\eps_1$ correspond to the smoothing of $(\nu,m_1,m_2)$ and $\eps_2$ correspond to smoothing of $(\sigma,\overline{m}_1,\overline{m}_2)$. We can then compute
\begin{equation}
\begin{aligned}
\diverg_x&\left(\left\langle\nut^{\eps_1}\otimes\sut^{\eps_2},q(\lambda,\mu)\right\rangle+m_2^{\eps_1}+\overline{m}_2^{\eps_2}\right)\\
& = \int_{\R}\diverg_x\left(\int_{\R}q(\lambda,\mu)\ \mathrm{d}\nut^{\eps_1}\right)\ \mathrm{d}\sut^{\eps_2} + \int_{\R}\int_{\R}q(\lambda,\mu)\ \mathrm{d}\nut^{\eps_1}\ \mathrm{d}(\partial_x\sut^{\eps_2})+\partial_x m_2^{\eps_1}+\partial_x\overline{m}_2^{\eps_2}\\
&= \int_{\R}\langle\partial_x\nut^{\eps_1},q(\lambda,\mu)\rangle\ \mathrm{d}\sut^{\eps_2} + \int_{\R}\langle\partial_x\sut^{\eps_2},q(\lambda,\mu)\rangle\ \mathrm{d}\nut^{\eps_1}+\partial_x m_2^{\eps_1}+\partial_x\overline{m}_2^{\eps_2}
\end{aligned}
\end{equation}
where we have used Fubini's theorem in the last line. The tensor product $\nut\otimes\sut$ is defined as the product measure $\mathrm{d}\nut(\lambda)\mathrm{d}\sut(\mu)$. Similarly
\begin{equation}
\begin{aligned}
\partial_t&\left(\left\langle\nut^{\eps_1}\otimes\sut^{\eps_2},|\lambda-\mu|\right\rangle+m_1^{\eps_1}+\overline{m}_1^{\eps_2}\right)\\
&= \int_{\R}\langle\partial_t\nut^{\eps_1},|\lambda-\mu|\rangle\ \mathrm{d}\sut^{\eps_2} + \int_{\R}\langle\partial_t\sut^{\eps_2},|\lambda-\mu|\rangle\ \mathrm{d}\nut^{\eps_1}+\partial_t m_1^{\eps_1}+\partial_t\overline{m}_1^{\eps_2}.
\end{aligned}
\end{equation}
Consequently using ~\eqref{smoothentropynu}
\begin{equation}
\begin{aligned}
\int_{\tplus}&\left(\langle\nut^{\eps_1}\otimes\sut^{\eps_2},|\lambda-\mu|\rangle + m_1^{\eps_1}+\overline{m}_1^{\eps_2}\right)\partial_t\phi\ \dxdt\\
&+ \int_{\tplus}\left(\langle\nut^{\eps_1}\otimes\sut^{\eps_2},q(\lambda,\mu)\rangle + m_2^{\eps_1}+\overline{m}_2^{\eps_2}\right)\nabla\phi\ \dxdt\\
&\hspace{-0.5cm}=-\int_{\tplus}\phi\int_{\R}\left\{\left(\langle\partial_t\nut^{\eps_1},|\lambda-\mu|\rangle + \langle\partial_x\nut^{\eps_1},q(\lambda,\mu)\rangle\right)\ \mathrm{d}\sut^{\eps_2}-\partial_t m_1^{\eps_1}-\partial_x  m_2^{\eps_1}\right\}\ \dxdt \\
&-\int_{\tplus}\phi\int_{\R}\left\{\left(\langle\partial_t\sut^{\eps_2},|\lambda-\mu|\rangle + \langle\partial_x\sut^{\eps_2},q(\lambda,\mu)\rangle\right)\ \mathrm{d}\nut^{\eps_1}-\partial_t \overline{m}_1^{\eps_2}-\partial_x  \overline{m}_2^{\eps_2}\right\}\ \dxdt\\
&\hspace{-0.5cm}\geq 0.
\end{aligned}
\end{equation}
This holds for any $\phi\in \mathrm{C}_c^\infty(V)$, thus establishing the following inequality in $\mathcal{D}'(\tplus)$ 
\begin{equation}\label{smoothentropy}
\partial_t\left(\left\langle\nut^{\eps_1}\otimes\sut^{\eps_2},|\lambda-\mu|\right\rangle + m_1^{\eps_1} + \overline{m}_1^{\eps_2}\right) + \diverg_x\left(\left\langle\nut^{\eps_1}\otimes\sut^{\eps_2},q(\lambda,\mu)\right\rangle + m_2^{\eps_1} + \overline{m}_2^{\eps_2}\right)\leq 0.
\end{equation}
Observe that the function $\varphi(x,t) = \psi(x)\phi(t)$ with $\psi\equiv 1$ and $\phi\in \mathrm{C}_c^\infty(\R_+)$ is an admissible test function, since $\T$ is a compact manifold without boundary. Testing ~\eqref{smoothentropy} with such a function yields
\begin{equation}
\int_{\tplus}\left(\langle\nut^{\eps_1}\otimes\sut^{\eps_2},|\lambda-\mu|\rangle + m_1^{\eps_1} + \overline{m}_1^{\eps_2}\right)\partial_t\phi\ \dxdt \geq 0.
\end{equation}
The argument presented in \cite{Szepessy} establishes the limit
\[\lim\limits_{\eps_1\to 0}\:\lim\limits_{\eps_2\to 0}\ \int_{\tplus}\left\langle\nut^{\eps_1}\otimes\sut^{\eps_2},|\lambda-\mu|\right\rangle\partial_t\phi\ \dxdt = \int_{\tplus}\left\langle\nut\otimes\sut,|\lambda-\mu|\right\rangle\ \dxdt.
\]
Clearly we have the convergence
\[\int_{\tplus}m_1^{\eps_1}\partial_t\phi\ \dxdt = \int_{\tplus}\partial_t(\phi*\eta^{\eps_1})\ m_1(\dd)\xrightarrow{\eps_1\to 0}\int_{\tplus}\partial_t\phi\ m_1(\dd).
\]
Similarly $\int\overline{m}_1^{\eps_2}\partial_t\phi \rightarrow\int\partial_t\phi\ \overline{m}_1(\dd)$.
We thus have for any non-negative $\phi\in \mathrm{C}_c^\infty(\R_+)$
\begin{equation}\label{avcontraction}
\int_{\tplus}\left\langle\nut\otimes\sut,|\lambda-\mu|\right\rangle\partial_t\phi(t)\ \dxdt +\int_{\tplus}\partial_t\phi\ m_1(\dd) + \int_{\tplus}\partial_t\phi\ \overline{m}_1(\dd)\geq 0.
\end{equation}
Let now 
\begin{equation}\label{defA}
A(t) \coloneqq \int_{\T}\langle\nut\otimes\sut,|\lambda-\mu|\rangle\ \dx.
\end{equation}
Then $A$ is a non-negative locally integrable function on $\R_+$. Let $\tau>0$ be a fixed Lebesgue point of $A$ and define $\phi:\R_+\to\R_+$ by
\[\phi(t) = \left(\frac{t}{\eps}-1\right)\chi_{(\eps,2\eps)}(t) + \chi_{(2\eps,\tau)}(t) + \left(-\frac{t}{2\eps}+\frac{\tau+\eps}{2\eps}\right)\chi_{(\tau-\eps,\tau+\eps)}(t).\]
We denote by $\phi^\delta$, $\delta<\eps$, a mollification of $\phi$ by a smooth kernel. Likewise $m_1^\delta$ and $\overline{m}_1^\delta$ denote mollifications of measures $m_1$ and $\overline{m}_1$ with respect to the time variable. Notice that $\partial_t\phi$ is supported only in the intervals $(\eps,2\eps)$ and $(\tau-\eps,\tau+\eps)$, where it is equal to $\frac{1}{\eps}$ and $-\frac{1}{2\eps}$ respectively. Hence using the following inequalities
\[m_1(\T\times(a,b)) \leq \liminf\limits_{\delta \to 0^{+}}\int_{\R_+}\Phi\ m_1^{\delta}(\dd) \leq \limsup\limits_{\delta \to 0^{+}} \int_{\R_+} \Phi\ m_1^{\delta}(\dd) \leq  m_1(\T\times[a,b]),
\]
where $\Phi = \chi_{(a,b)}$ a.e. with respect to the Lebesgue measure, we obtain using also ~\eqref{avcontraction}
\begin{equation}\label{approxA}
\begin{aligned}
\frac{1}{\eps}&\int_\eps^{2\eps}A(t)\ \dt + \frac{1}{\eps}\int_{\T\times[\eps,2\eps]}\left(m_1(\dd)+\overline{m}_1(\dd)\right)\\
&\hspace{0.5cm}-\frac{1}{2\eps}\int_{\tau-\eps}^{\tau+\eps}A(t)\ \dt - \frac{1}{2\eps}\int_{\T\times(\tau-\eps,\tau+\eps)}\left(m_1(\dd)+\overline{m}_1(\dd)\right) \\
&\geq \lim\limits_{\delta\to 0}\;\int_{\tplus}A(t)\partial_t\phi^\delta\ \dt + \int_{\tplus}\partial_t\phi^\delta\left(m_1(\dd)+\overline{m}_1(\dd)\right) \\
&\geq 0.
\end{aligned}
\end{equation}
This implies that
\begin{equation}\label{approxAbound}
\begin{aligned}
\frac{1}{2\eps}&\int_{\tau-\eps}^{\tau+\eps}A(t)\ \dt + \frac{1}{2\eps}\int_{\T\times(\tau-\eps,\tau+\eps)}\left(m_1(\dd)+\overline{m}_1(\dd)\right) \\
&\leq \frac{1}{\eps}\int_\eps^{2\eps}A(t)\ \dt + \frac{1}{\eps}\int_{\T\times[\eps,2\eps]}\left(m_1(\dd)+\overline{m}_1(\dd)\right) \\
&\leq \frac{1}{\eps}\int_0^{2\eps}A(t)\ \dt + \frac{1}{\eps}\int_{\T\times[0,2\eps]}\left(m_1(\dd)+\overline{m}_1(\dd)\right).
\end{aligned}
\end{equation}
Furthermore
\begin{equation*}
\begin{aligned}
A(t)&\leq\int_{\T}\langle\nut\otimes\sut,|\lambda-u_0(x)|\rangle\ \dx + \int_{\T}\langle\nut\otimes\sut,|\mu-u_0(x)|\rangle\ \dx\\
& =\int_{\T}\langle\nut,|\lambda-u_0(x)|\rangle\ \dx + \int_{\T}\langle\sut,|\mu-u_0(x)|\rangle\ \dx.
%A(t)& + \int_{\tplus}\left(m_1(\dd)+\overline{m}_1(\dd)\right)\\
%&\leq \int_{\T}\langle\nut\otimes\sut,|\lambda-u_0(x)|\rangle\ \dx + \int_{\tplus}m_1(\dd) \\
%&\hspace{0.5cm}+\int_{\T}\langle\nut\otimes\sut,|\mu-u_0(x)|\rangle\ \dx + \int_{\tplus}\overline{m}_1(\dd)\\ 
%& = \left\{\int_{\T}\langle\nut,|\lambda-u_0(x)|\rangle\ \dx + \int_{\tplus}m_1(\dd)\right\} \\
%&\hspace{0.5cm}+\left\{\int_{\T}\langle\sut,|\mu-u_0(x)|\rangle\ \dx + \int_{\tplus}\overline{m}_1(\dd)\right\}\\
\end{aligned}
\end{equation*}
Consequently by the initial condition ~\eqref{mildcts} we have
\[
\lim\limits_{T\to 0^+}\left\{\frac{1}{T}\int_0^T A(t)\ \dt + \frac{1}{T}\int_{\T\times[0,T]}m_1(\dd) + \frac{1}{T}\int_{\T\times[0,T]}\overline{m}_1(\dd)\right\} = 0.
\]
It follows that the right hand side of ~\eqref{approxAbound} converges to zero as $\eps\to0$. Therefore, since $A$ and the measures $m_1$, $\overline{m}_1$ are non-negative, we see that
\begin{equation}\label{limitsA}
\lim\limits_{\eps\to 0}\;\frac{1}{2\eps}\int_{\tau-\eps}^{\tau+\eps}A(t)\ \dt = 0,\hspace{0.25cm}\text{and}\hspace{0.25cm}\lim\limits_{\eps\to 0}\;\frac{1}{2\eps}\int_{\T\times(\tau-\eps,\tau+\eps)}\left(m_1(\dd)+\overline{m}_1(\dd)\right) = 0.
\end{equation}
Hence, by Lebesgue differentiation theorem, it follows that $A(\tau) = 0$ for a.a. $\tau\in\R_+$.
This implies that
\begin{equation}\label{supports}
\int_{\R\times\R}|\lambda-\mu|\ \mathrm{d}\nut(\lambda)\mathrm{d}\sut(\mu) = 0.
\end{equation}
It can be easily seen from ~\eqref{supports} that the measures $\nut$ and $\sut$ have a common support consisting of a point ${w(x,t)}$ for a.e. $(x,t)\in\tplus$. Further, the second limit in ~\eqref{limitsA} implies that $m_1 = \overline{m}_1 = 0$. To see this consider an arbitrary time interval $[a,b]\subset\R_+$. It can be covered with a finite number of overlapping open intervals of radius $\eps$, denoted $B(t_i,\eps)$, so that
\[\sum\limits_{i\in J}\mathcal{L}^1(B(t_i,\eps))\leq 2(b-a),
\]
where $\mathcal{L}^1$ denotes the $1$-dimensional Lebesgue measure. From ~\eqref{limitsA} we have that
\[
m_1(\T\times B(\tau,\eps))\leq C(\eps)\mathcal{L}^{d+1}(\T\times B(\tau,\eps))\]
for any ball, where $C(\eps)\to 0$ as $\eps\to 0$. Hence
\[m_1(\T\times[a,b])\leq C(\eps)\sum\limits_{i\in J}\mathcal{L}^{d+1}(\T\times B(t_i,\eps))\leq 2C(\eps)\mathcal{L}^{d}(\T)(b-a)\xrightarrow{\eps\to 0} 0.
\]
Hence $m_1(\T\times [a,b]) = 0$ for any $a<b$. Similarly $\overline{m}_1 = 0$. 
This in turn implies $m_2 = \overline{m}_2=0$ as a consequence of ~\eqref{mmbound}.\\
Finally the claimed regularity of $w$ follows from the convergence
\[
\int_{\tplus}g(w_k(x,t))\cdot\phi(x,t)\ \dxdt \xrightarrow{k\to\infty}\int_{\tplus}\underbrace{\langle\nut,g\rangle}_{\text{$g(w(x,t))$}}\phi(x,t)\ \dxdt
\]
and the boundedness of the sequence $u_k$.
\end{proof}
\begin{remark}
One can also pass to the limit in the divergence term of ~\eqref{smoothentropy} to obtain the following averaged contraction principle
\begin{equation}\label{averagedcontraction}
\partial_t\left(\left\langle\nut\otimes\sut,|\lambda-\mu|\right\rangle + m_1 + \overline{m}_1\right) + \diverg_x\left(\left\langle\nut\otimes\sut, q(\lambda,\mu)\right\rangle + m_2 + \overline{m}_2\right)\leq 0.
\end{equation}
The theorem is extendible to the problem defined on $\R^d\times\R_+$ rather than $\tplus$. An approximation argument is needed in that case since the constant unit function no longer belongs to the class $\mathrm{C}_c^\infty(\R^d)$. 
\end{remark}
We have thus established uniqueness in the class of mv entropy solutions. We conclude this section with proving existence of a unique entropy solution.
First we observe that recession functions for  $\eta = |\lambda-\mu|$ and $q = \sgn(\lambda-\mu)(f(\lambda)-f(\mu))$ can be easily described as follows:
\begin{equation}\label{recessioneta}
\eta^\infty = \lim\limits_{\lambda\to\pm\infty}\frac{|\lambda v-k|}{|\lambda|} = 1
\end{equation}
where $v\in \mathbb{S}^0$, and
\begin{equation}\label{recessionq}
q^\infty(v) =
\begin{cases}
f^\infty(1), & \text{for}\; v=1,\\
-f^\infty(-1), & \text{for}\; v=-1,
\end{cases}
\end{equation}
where we define 
\begin{equation}\label{finfty}
f^\infty(1)\coloneqq\lim\limits_{\lambda\to+\infty}\frac{f(\lambda)}{\lambda}\hspace{0.25cm}\text{and}\hspace{0.25cm}f^\infty(-1)\coloneqq\lim\limits_{\lambda\to-\infty}\frac{f(-\lambda)}{\lambda}
\end{equation}
Notice that in fact the assumption on existence of the limits in ~\eqref{finfty} together with continuity of the flux implies condition ~\eqref{sublinear}.
\begin{theorem}
There exists a unique entropy solution $w$ to ~\eqref{conservLaw} such that
\begin{equation}
\norm{w(\cdot,t)}_{L^1(\T)}\leq\norm{u_0}_{L^1(\T)}
\end{equation}
for a.e. $t\in\R_+$.
\end{theorem}
\begin{proof}
We adapt the proof presented in ~\cite{Szepessy}. First we prove existence of a mv entropy solution $(\nu,m_1,m_2)$ - to this end we consider the following parabolic regularization
\begin{equation}\label{regularproblem}
\begin{aligned}
\partial_t w_n+\diverg_x f_n(w_n) &= \frac{1}{n}\Delta w_n,\hspace{0.2cm} \text{on}\; \tplus,\\
w_n(x,0) &= u_0^n(x),\hspace{0.25cm} \text{on}\; \T
\end{aligned}
\end{equation}
where $u_0^n\in \mathrm{C}_c^\infty(\T)$ and $u_0^n\rightarrow u_0$ in $L^1(\T)$ and $f_n\coloneqq f*\eta^{\eps_n}$. The sequence $\eps_n$ is chosen so that for $|z|\leq\norm{u_0}_{L^\infty}$ we have
\begin{equation}\label{nbound}
\sup\limits_{|h|\leq\eps_n}|f(z+h)-f(z)|\leq\frac{1}{n}.
\end{equation}
This is possible because of the assumptions on asymptotic behaviour ~\eqref{finfty} of the flux.
Problem ~\eqref{regularproblem} has a unique smooth solution $w_n$ satisfying the bound $\norm{w_n(\cdot,t)}_{L^1(\T)}\leq\norm{u_0}_{L^1(\T)}$ for a.e. $t$. Since the sequence $\{w_n\}$ is bounded in $L^\infty(\R_+;L^1(\T))$, there is, after passing to a subsequence, a triple $(\nu,m_1,m_2)$ of associated Young measure and concentration measures. We then have by ~\eqref{nbound}
\begin{equation}
\begin{aligned}
\lim\limits_{n\to\infty}&\;\int_{\tplus}\sgn(w_n-k)(f_n(w_n)-f_n(k))\nabla\phi\ \dxdt\\
&=\lim\limits_{n\to\infty}\;\int_{\tplus}\sgn(w_n-k)\left[(f(w_n)-f(k)) + (f_n(w_n)-f(w_n)) + (f_n(k) - f(k)) \right]\nabla\phi\ \dxdt\\
&=\lim\limits_{n\to\infty}\;\int_{\tplus}\sgn(w_n-k)(f(w_n)-f(k))\nabla\phi\\
&=\int_{\tplus}\left(\langle\nut, \sgn(\lambda-k)(f(\lambda)-f(k))\rangle+m_2\right)\nabla\phi\ \dxdt,
\end{aligned}
\end{equation}
where the last equality follows from Lemma\ref{measureschar}.
%In the above calculation we have used the decomposition $f_n = (f_n-f) + f$. We then observe that $f_n-f\in \mathrm{C}_b(\R)$ and $\norm{f_n-f}_{\infty}\xrightarrow{n\to\infty} 0$.
Moreover, by virtue of the same lemma, we have
\begin{equation}
\lim\limits_{n\to\infty}\;\int_{\tplus}|w_n-k|\partial_t\phi\ \dxdt = \int_{\tplus}\left(\langle\nut,|\lambda-k|\rangle+m_1\right)\partial_t\phi\ \dxdt.
\end{equation}
We observe that from Lemma\ref{measureschar} and ~\eqref{recessioneta}-~\eqref{finfty}, the following characterization of the concentration measures holds true
\[
m_2 = \left(\int_{\mathbb{S}^0}q^\infty(v)\ \mathrm{d}\nut^\infty\right)m_1.
\]
Clearly this implies the bound ~\eqref{mmbound}. We remark that the measures $m_1$ and $m_2$ are common for each choice of entropy-entropy flux pair (i.e. for each choice of $k\in\R$), because the associated recession functions do not depend on $k$.\\
Let now $\sgn_\delta$, $\theta_\delta$ and $q_\delta$ be regularizations of the functions $z\mapsto\sgn(z)$, $|z-k|$ and $q(z,k)$ respectively, for $k\in\R$. Multiplying ~\eqref{regularproblem} by $\phi\sgn_\delta(w_n-k)$, integrating in time and space and integrating by parts yields
\begin{equation}\label{regulartested}
\begin{aligned}
\int_{\tplus}&\left(\theta_\delta(w_n,k)\partial_t\phi + q_\delta(w_n,k)\nabla\phi\right)\ \dxdt\\
&= \frac{1}{n}\int_{\tplus}|\nabla w_n|^2\sgn'_\delta(w_n-k)\phi\ \dxdt - \frac{1}{n}\int_{\tplus}\theta_\delta(w_n,k)\Delta\phi\ \dxdt.\\
\end{aligned}
\end{equation}
The right hand side of the last identity can be bounded from below by 
\begin{equation}\label{boundontested}
-\frac{C(\phi)}{n}\left(\norm{w_n}_{L^\infty(R_+;L^1(\T))} + 1 \right).
\end{equation}
Combining ~\eqref{regulartested} and ~\eqref{boundontested} and passing with $\delta$ to zero we see that
\[
\int_{\tplus}|w_n-k|\partial_t\phi + \sgn(w_n-k)(f_n(w_n)-f_n(k))\nabla\phi\ \dxdt \geq - \frac{C}{n}.
\]
Then passing to infinity with $n$ we get
\[
\partial_t\left(\langle\nut,|\lambda-k|\rangle+m_1\right) + \diverg_x\left(\langle\nut,q(\lambda,k)+m_2\right)\leq 0.
\]
Therefore the generalized Young measure generated by the sequence $\{w_n\}$ satisfies the entropy inequlity ~\eqref{entropycond}. Standard methods of analysis for PDEs can be employed to show that the initial condition ~\eqref{mildcts} is satisfied as well.
%Similarly one can adapt the proof in \cite{Szepessy} to show that the initial condition ~\eqref{mildcts} is satisfied.
Therefore $(\nu,m_1,m_2)$ is a mv entropy solution of ~\eqref{conservLaw}. By the previous theorem there is a function $w$ such that $\nut = \delta_{\{w(x,t)\}}$. This function is then the unique entropy solution.
\end{proof}
%\begin{remark}
%The condition ~\eqref{finfty} of existence of $f^\infty$ is superfluous. It is in fact only required by the statement of Lemma~\eqref{measureschar}. One can obtain the bound ~\eqref{mmbound} solely from the growth condition ~\eqref{sublinear}.  
%\end{remark}
\begin{remark}
Other approximation schemes can be used rather than the viscosity approximation employed here. Indeed from the point of view of numerics other schemes are favourable. In \cite{LiPeTa}, \cite{PeTa}, \cite{PeTz} and \cite{BuGG} a kinetic approximation (i.e. an approximation by a Boltzmann type equation) is used, while in \cite{Chen} a hyperbolic conservation law is realized as a limit of the attracitve zero range process (ZRP). In the latter paper a discontinuous flux is considered. An extension of the averaged contraction principle to the case of discontinuous flux (both in $x$ and $u$) was considered in a series of papers \cite{BuGG, BuGG2, BuGG3, GwSwWitZim}.
\end{remark}

\section{Relative entropy method for equations of fluid dynamics}\label{sec3}
In this section we survey selected weak-strong uniqueness results for equations of fluid dynamics. A global existence of measure-valued solutions to the incompressible Euler system was proven in \cite{DiMa} for any finite-energy initial data. Later existence was also shown for the compressible Euler and Navier-Stokes systems, cf. \cite{Neustup}, \cite{KroZaj}. The first weak-strong uniqueness result was proven for incompressible Euler in \cite{BrDLSz}. Note however that existence of a strong solution is needed -- otherwise uniqueness for admissible solutions might not hold, cf. \cite{DLSz}, \cite{SzWi}.
\subsection{Euler equations}
First we consider the \emph{incompressible Euler equations}
\begin{equation}\label{inEuler}
\begin{aligned}
\partial_t u + \diverg(u\otimes u)+\nabla p&=0,\\
\diverg u&=0,
\end{aligned}
\end{equation}
where $u:\tplus\to\R^d$ is the velocity of a fluid and $p$ is the scalar pressure.
\begin{definition}
Let $\nu$ be a Young measure, $m$ a matrix-valued measure on $\Tbb^d\times [0,T]$ satisfying $m(\dxdt) = m_t(\dx)\otimes\dt$ for some family $\{m_t\}_{t\in(0,T)}$ of uniformly bounded measures on $\R^d$. Further let $D\in L^{\infty}(0,T)$ with $D\geq 0$ such that $|m_t|(\Tbb^d)\leq CD(t)$ for some constant $C>0$ and almost every $t\in [0,T]$.\\
The triple $(\nu, m, D)$ is called a \emph{dissipative measure-valued solution} to ~\eqref{inEuler} with initial datum $u_0$ if it satisfies the following conditions:
\begin{enumerate}
    \item for any divergence-free $\phi\in \mathrm{C}^1(\Tbb^d\times[0,T];\R^d)$ the equation
    \begin{equation}
    \begin{aligned}
    \int_0^T\int_{\Tbb^d}\partial_t\phi(x,t)&\cdot\langle\nut,\id\rangle + \nabla\phi(x,t):(\langle\nut,\id\otimes\id\rangle+m_t)\ \dxdt\\  
    &=\int_{\T}\langle\nutt,\id\rangle\cdot\phi(x,\tau)-u_0(x)\cdot\phi(x,0)\ \dx
    \end{aligned}
    \end{equation}
    holds for almost every $\tau\in (0,T)$;
    \item the divergence free condition
    \[\int_{\T}\langle\nut,\id\rangle\cdot\nabla\psi(x)\ \dx = 0
    \]
    holds for every $\psi\in \mathrm{C}^1(\T)$ and a.e. $t\in(0,T)$;
    \item the admissibility condition
    \[E(\tau)\leq\frac{1}{2}\int_{\T}|u_0(x,\tau)|^2\ \dx
    \]
    is satisfied for a.e $\tau\in(0,T)$, where the measure-valued energy is defined by
    \[
    E(\tau)\coloneqq \frac{1}{2}\int_{\T}\langle\nutt,|u|^2\rangle\ \dx + D(\tau).
    \]
\end{enumerate}
\end{definition}
We will now show how measuring the relative entropy between a dissipative mvs and a strong solution leads to a uniqueness result. We repeat the proof presented in \cite{Wiedemann}.
\begin{theorem}
Let $(\nu,m,D)$ be a dissipative measure-valued solution and $U\in \mathrm{C}^1(\Tbb^d\times[0,T])$ a strong solution to ~\eqref{inEuler} with the same initial datum $u_0$. Then $\nutt = \delta_{\{U(x,\tau)\}}$ for almost every $(x,\tau)\in\Tbb^d\times(0,T)$, $m=0$ and $D=0$.
\end{theorem}
\begin{proof}
We begin by defining the relative entropy $E_{\text{rel}}$ as
\[E_{rel}(\tau) \coloneqq \frac{1}{2}\int_{\T}\langle\nutt,|\id - U(x,\tau)|^2\rangle\ \dx + D(\tau).
\]
This quantity can be estimated as follows.
\begin{equation*}
    \begin{aligned}
    E_{\text{rel}}(\tau) &= \frac{1}{2}\int_{\T}|U(x,\tau)|^2\ \dx + \frac{1}{2}\int_{\T}\langle\nutt,|\id-U|^2\rangle\ \dx - \int_{\T}\langle\nutt,\id\rangle\cdot U(x,\tau)\ \dx + D(\tau) \\
    &\leq \frac{1}{2}\int_{\T}|u_0|^2\ \dx + \frac{1}{2}\int_{\T}|u_0|^2\ \dx - \int_{\T}\langle\nutt,\id\rangle\cdot U(x,\tau)\ \dx \\
    &=\intd |u_0|^2\ \dx - \intd u_0(x)\cdot u_0(x)\ \dx \\
    &\hspace{1cm}- \int_0^{\tau}\intd \partial_t U(x,t)\cdot\langle\nut,\id\rangle + \nabla U(x,t):\langle\nut,\id\otimes\id\rangle \ \dxdt \\
    &\hspace{1cm}-\int_0^{\tau}\intd \nabla U(x,t)\ \mathrm{d}m(x,t)\\
    &=\int_0^{\tau}\intd\diverg(U\otimes U)(x,t)\cdot \langle\nut,\id\rangle - \nabla U(x,t)\langle\nut,\id\otimes\id\rangle \ \dxdt \\
    &\hspace{1cm}-\int_0^{\tau}\intd\nabla_{\text{sym}}U(x,t)\ \mathrm{d}m(x,t) \\
    &=\int_0^{\tau}\intd \langle\nut, (U(x,t)-\id)\cdot\nabla_{\text{sym}}U(x,t)(\id - U(x,t))\rangle\ \dxdt \\
    &\hspace{1cm} - \int_0^{\tau}\intd\nabla_{\text{sym}}U(x,t)\ \mathrm{d}m(x,t) \\
    &\leq\int_0^{\tau}\norm{\nabla_{\text{sym}}U(t)}_{L^{\infty}}E_{\text{rel}}(t)\ \dt.
    \end{aligned}
\end{equation*}
It now follows from Gronwall's inequality that the relative entropy is zero almost everywhere. This in turn implies that $\nutt = \delta_{\{U(x,\tau)\}}$ and $D(\tau) = 0$.  
\end{proof}
Now consider the {\it isentropic compressible Euler system}
\begin{equation}\label{compresseuler}
\begin{aligned}
\partial_t\rho + \diverg(\rho u) &=0,\\
\partial_t(\rho u) + \diverg(\rho u\otimes u) + \nabla(\rho^\gamma) &= 0,
\end{aligned}
\end{equation}
where $\gamma>1$ is the adiabatic coefficient. The definition of a measure-valued solution to the above system requires a slight refinement of the Alibert-Bouchitt\'{e} framework, which we ignore here; see Section $3$ in \cite{GwSwWi} for details. The following notation is used for brevity
\[\overline{f}(\dxdt) = \langle\nut,f\rangle\;\dxdt+\langle\nut^\infty,f^\infty\rangle\;m(\dxdt).
\]
\begin{definition}
The triple $(\nu,m,\nu^\infty)$ is called a \emph{measure-valued solution} of ~\eqref{compresseuler} with initial data $(\rho_0,u_0)$ such that $\rho_0$ and $\rho_0 u_0$ are integrable if for every $\tau\in[0,T]$, $\psi\in \mathrm{C}^1(\T \times [0,T])$ and $\phi\in \mathrm{C}^1(\T \times [0,T];\R^n)$
\begin{equation}
\int_0^T\int_{\T}\partial_t\psi\overline{\rho}+\nabla\psi\cdot\overline{\rho }\ \dxdt + \int_{\T}\psi(x,0)\rho_0 - \psi(x,T)\overline{\rho}(x,T)\ \dx = 0
\end{equation}
and
\begin{equation}
\begin{aligned}
\int_0^T\int_{\T}\partial_t\phi\cdot\overline{\rho u} &+ \nabla\phi:\overline{\rho u\otimes u} + \diverg\phi\overline{\rho^\gamma}\ \dxdt \\
&+\int_{\T}\phi(x,0)\cdot\rho_0 u_0 - \phi(x,T)\cdot\overline{\rho u}(x,T)\ \dx = 0.
\end{aligned}
\end{equation}
We then define the entropy of such a measure-valued solution
\[
E_{\text{mvs}}(t)\coloneqq \int_{\T}\frac{1}{2}\overline{\rho |u|^2}(x,t) + \frac{1}{\gamma-1}\overline{\rho^\gamma}(x,t)\ \dx
\]
and
\[
E_0 \coloneqq \int_{\T}\frac{1}{2}\rho_0 |u_0|^2(x) + \frac{1}{\gamma-1}\rho_0^\gamma(x)\ \dx
\]
\end{definition}
An admissibility criterion is then posed as follows
\[
E_{\text{mvs}}\leq E_0.
\]
It was shown in \cite{GwSwWi} that one can use the relative entropy method to prove the following weak-strong uniqueness result.
\begin{theorem}
Suppose 
\[(P,U)\in W^{1,\infty}(\T \times [0,T])\times C^1(\T \times [0,T])
\]
is a strong solution of ~\eqref{compresseuler} with initial data $(\rho_0,u_0)$ s.t.
\[
\rho_0\geq c>0,\hspace{0.5cm}\rho_0\in L^r(\T),\hspace{0.5cm}\rho_0 u_0\in L^1(\T)\hspace{0.25cm}\text{and}\hspace{0.25cm}P\geq c>0.
\]
Then if $(\nu,m,\nu^\infty)$ is an admissible mv solution with the same initial data, then
\[\nut = \delta_{\left(P(x,t),\sqrt{P(x,t)}U(x,t)\right)}\hspace{0.25cm}\text{for a.e.}\; (x,t)
\]
and $m=0$.
\end{theorem}
The relative entropy functional used in this case has the form
\[
E_{\text{rel}}(t)\coloneqq \frac{1}{2}\int_{\T}\overline{\rho|u-U|^2} + \frac{1}{\gamma-1}\overline{\rho^\gamma}-\frac{\gamma}{\gamma-1}\overline{P^{\gamma-1}\rho} + \overline{P^\gamma}\ \dx.
\]

\subsection{Navier-Stokes equations}
We now consider the barotropic Navier-Stokes system
\begin{eqnarray}
\label{I1}
\partial_t \vr + \Div (\vr u) &=& 0,  \\
\label{I2}  \partial_t (\vr u) + \Div (\vr u \otimes u) + \Grad p(\vr) &=& \Div \mathbb{S} (\Grad u), \\
\label{I3} u|_{\partial \Omega} &=& 0.
\end{eqnarray}
Here $\Omega$ is a regular bounded domain in $\R^2$ or $\R^3$ and $\mathbb{S}$ is the Newtonian viscous stress.
The following definition of a dissipative measure-valued solution and subsequent results are taken from \cite{FeGSG}.
\begin{definition} \label{DD1}
We say that a parametrized measure $\{ \nut \}_{(x,t) \in \Omega \times (O,T)}$,
\[
\nu \in L^{\infty}_{\rm weak}\left(  \Omega \times (0,T); \mathcal{P} \left( \mathbb{R}^N \times [0,\infty)\right) \right),\
\left< \nut; s \right> \equiv \vr,\ \left< \nut; \vc{v} \right> \equiv u
\]
is a dissipative measure-valued solution of the Navier-Stokes system (\ref{I1} -- \ref{I3}) in
$\Omega \times (0,T)$, with the initial conditions $\nu_{(x,0)}$ and dissipation defect $\mathcal{D}$,
\[
\mathcal{D} \in L^\infty(0,T), \ \mathcal{D} \geq 0,
\]
if the following holds.
\begin{itemize}
\item {\bf Equation of continuity.}
There exists a measure $r^C\in L^1(0,T;\mathcal{M}(\Ov{\Omega}))$ and $\chi\in L^1(0,T)$ such that for a.a.\ $\tau\in(0,T)$ and every $\psi \in \mathrm{C}^1(\Ov{\Omega} \times [0,T])$,
\begin{equation}
\left| \langle r^C (\tau) ; \Grad\psi \rangle \right| \leq \chi(\tau)  \mathcal{D} (\tau) \| \psi \|_{C^1(\Ov{\Omega})}
\end{equation}
and
\begin{equation} \label{dmvB1}
\begin{split}
\intO{ \langle \nutt; s \rangle \psi (\cdot, \tau) } &-  \intO{ \langle \nu_{0}; s \rangle \psi (0, \cdot) } \\
&= \int_0^\tau \intO{ \Big[ \langle \nut; s \rangle \partial_t \psi + \langle \nut; s \vc{v} \rangle \cdot \Grad \psi \Big] } \ \dt
+ \int_0^\tau \langle r^C; \Grad \psi \rangle \ \dt.
\end{split}
\end{equation}

\color{black}

\item {\bf Momentum equation.}
\[
u = \left< \nut; \vc{v} \right> \in L^2(0,T; W^{1,2}_0 (\Omega;\mathbb{R}^N)),
\]
and there exists a measure $r^M\in L^1(0,T;\mathcal{M}(\Ov{\Omega}))$ and $\xi\in L^1(0,T)$ such that for a.a.\ $\tau\in(0,T)$ and every $\varphi \in \mathrm{C}^1(\Ov{\Omega} \times [0,T] ; \mathbb{R}^N)$,  $\varphi|_{\partial \Omega} = 0$,
\begin{equation}
\left| \langle r^M (\tau) ; \Grad\varphi \rangle \right| \leq \xi(\tau)  \mathcal{D} (\tau) \| \varphi \|_{C^1(\Ov{\Omega})}
\end{equation}
and
\begin{equation} \label{dmvB2}
\begin{split}
&\intO{ \langle \nutt; s \vc{v} \rangle \cdot \varphi ( \cdot, \tau) }  -  \intO{ \langle \nu_{0}; s \vc{v} \rangle \cdot \varphi (0, \cdot) }\\
&= \int_0^\tau \intO{ \Big[ \langle \nut ; s \vc{v} \rangle \cdot \partial_t \varphi + \langle \nut; s (\vc{v} \otimes \vc{v} ) \rangle : \Grad \varphi +
\langle \nut ; p(s) \rangle \Div \varphi \Big] } \ \dt\\
& - \int_0^\tau \intO{  \mathbb{S}({\Grad u })  : \Grad \varphi } \ \dt + \int_0^\tau \left< {r}^M ; \Grad \varphi \right> \ \dt.
\end{split}
\end{equation}

\color{black}

\item{\bf Energy inequality.}
\begin{equation} \label{dmvEI}
\begin{split}
\intO{ \left< \nutt;  \left( \frac{1}{2} s |\vc{v}|^2 + P(s) \right) \right> }
&+ \int_0^\tau \intO{ \mathbb{S}(\Grad u) : \Grad u  } \ \dt + \mathcal{D}(\tau) \\
&\leq \intO{ \left< \nu_{(x,0)}; \left( \frac{1}{2} s |\vc{v}|^2 + P(s) \right) \right>  }
\end{split}
\end{equation}
for a.a. $\tau \in (0,T)$.
In addition, the following version
of ``Poincar\' e's inequality" holds for a.a. $\tau \in (0,T)$:
\begin{equation} \label{KoPo}
\int_0^\tau \intO{ \left< \nut ;  |\vc{v} - u|^2 \right> } \ \dt \leq c_P  \mathcal{D}(\tau).
\end{equation}
\end{itemize}
\end{definition}
One can show (Theorem 2.1 in \cite{FeGSG}) that if the pressure satisfies the following coercivity assumptions
\begin{equation} \label{BB6}
p \in \mathrm{C}[0, \infty) \cap C^2(0, \infty), \ p(0) = 0, \ p'(\vr) > 0 \ \mbox{for}\ \vr > 0, \ \liminf_{\vr \to \infty} p'(\vr) > 0,\
\liminf_{\vr \to \infty} \frac{P(\vr)}{p(\vr)} > 0,
\end{equation}
then there exists a dissipative mv solution with a prescribed finite-energy initial data. 
The following weak-strong uniqueness result can then be proven
\begin{theorem} \label{TT1}
Let $\Omega \subset \mathbb{R}^N$, $N=2,3$ be a bounded smooth domain. Suppose the pressure $p$ satisfies (\ref{BB6}). Let $\{ \nut, \mathcal{D} \}$ be a dissipative measure-valued
solution to the barotropic Navier-Stokes system (\ref{I1}--\ref{I3}) in $\Omega \times (0,T)$, with the initial state represented by $\nu_{(x,0)}$, in the sense specified in Definition \ref{DD1}.
Let $[r, \vc{U}]$ be a strong solution of (\ref{I1}--\ref{I3}) in $\Omega \times (0,T)$ belonging to the class
\[
r, \ \Grad r, \ \vc{U},\ \Grad \vc{U} \in \mathrm{C}(\Ov{\Omega} \times [0,T]),\
\partial_t \vc{U} \in L^2(0,T; C(\overline{\Omega};\R^N)),\ r > 0,\ \vc{U}|_{\partial \Omega} = 0.
\]

Then there is a constant $\Lambda = \Lambda(T)$, depending only on the norms of $r$, $r^{-1}$, $\vc{U}$, $\chi$, and $\xi$ in the aforementioned spaces, such that
\[
\begin{split}
& \intO{ \left< \nutt;  \frac{1}{2} s |\vc{v} - \vc{U}|^2 + P(s) - P'(r) (s - r) - P(r) \right> }
 \\ &+  \int_0^\tau \intO{ | \Grad u - \Grad \vc{U} |^2   } \ \dt + \mathcal{D}(\tau) \\
&\leq \Lambda(T) \intO{ \left< \nu_{(x,0)};  \frac{1}{2} s |\vc{v} - \vc{U}(0, \cdot) |^2 + P(s) - P'(r(0,\cdot)) (s - r(0,\cdot)) - P(r(0,\cdot)) \right> }
\end{split}
\]
for a.a. $\tau \in (0,T)$. In particular, if the initial states coincide, meaning
\[
\nu_{(x,0)} = \delta_{[ r(x,0), \vc{U}(x,0) ]} \ \mbox{for a.a.} \ x \in \Omega
\]
then $\mathcal{D} = 0$, and
\[
\nutt = \delta_{[ r(x, \tau), \vc{U}(x, \tau) ]} \ \mbox{for a.a.}\ \tau \in (0,T),\ x \in \Omega.
\]

\end{theorem}

\section{Polyconvex Elastodynamics}\label{sec4}

In this section we consider the {\it system of elasticity}
\begin{equation}
\label{mainI}
\frac{\partial^2 y}{\partial t^2}=\nabla\cdot S(\nabla y),
\end{equation}
where $y \; : \; {Q} \times {\R}^+ \to{\R}^3$ stands for the motion, $F = \nabla y$, $v = \partial_t y$,
and $S$ stands for the Piola-Kirchoff stress tensor
obtained as the gradient of a stored energy function, 
$S = \frac{\partial W}{\partial F}$.  Here we assume that
$W$ is polyconvex, that is  $W(F) = G ( \Phi(F))$ where 
$G:\mat^{3\times 3}\times\mat^{3\times 3}\times \R \to [0,\infty)$
is a strictly convex function and $\Phi(F) = (F ,\cof F, \det F)\in
\mat^{3\times 3}\times\mat^{3\times 3}\times \R$
stands for the vector of null-Lagrangians: $F$, the cofactor matrix $\cof F$
and the determinant $\det F$.
It is observed in \cite{DeStTz} and \cite{DeStTz2} that this system can be embedded into the following symmetrizable hyperbolic system in a new dependent variable $\Xi=(F,Z,w)$ taking values in
$\mat^{3\times 3}\times \mat^{3\times 3}\times\R$
\begin{equation}
\begin{aligned}
\frac{\partial v_i}{\partial t}&=\frac{\partial}{\partial x^\alpha}\left(\frac{\partial G}{\partial\Xi^A}(\Xi)\frac{\partial\Phi^A}{\partial
F_{i\alpha}}(F)\right),\\
\frac{\partial\Xi^A}{\partial t}&=\frac{\partial}{\partial
x^\alpha}\left(\frac{\partial\Phi^A}{\partial
F_{i\alpha}}(F)v_i\right).
\end{aligned}
\end{equation}
This system admits the following entropy-entropy flux pair
\begin{equation}
\begin{aligned}
\eta(v,F,Z,w)&=\frac{1}{2}|v|^2+G(F,Z,w),\\
q_\alpha&=v_i\,\frac{\partial G}{\partial\Xi^A}(\Xi)\frac{\partial\Phi^A}{\partial F_{i\alpha}}(F).
\end{aligned}
\end{equation}
A strong solution to ~\eqref{mainI} is a function $w\in W^{1,\infty}$. It automatically satisfies
\begin{equation}\label{elastodynentropy}
\partial_t \eta + \partial_\alpha q_\alpha = 0. 
\end{equation}
A weak solution which satisfies ~\eqref{elastodynentropy} as an inequality is called an entropy weak solution.
The following definition of a dissipative mv solution is taken from \cite{DeStTz}.

\begin{definition}
\label{defmvpcel}
Let the pair $(y,\nu)$ consist of a map
$y$, with distributional time and space derivatives 
$(v, F) \in L^{\infty}(L^2) \oplus L^{\infty}(L^p) $
and a Young measure $\nu=(\nut)_{x,t\in \overline{Q}_T}$ generated by a sequence satisfying
$$
\sup_{\eps, t}\int\,\eta( v^\eps , F^\eps, Z^\eps, w^\eps )\,dx\,<\infty
$$
which represents weak limits in the following way:
\begin{equation}
\label{defym}
\begin{aligned}
\hbox{wk-}\lim_{\epsilon\to 0}&\;f( v^\eps , F^\eps, Z^\eps, w^\eps )  = \int f(\lambda_v, \lambda_\Xi)  d\nut(\lambda_v, \lambda_\Xi) \\
&\qquad \forall \;\mbox{continuous}\; f=f(\lambda_v, \lambda_\Xi)\; \mbox{with }\; 
\lim_{|\lambda_v| + |\lambda_\Xi| \to \infty} 
\frac{ f(\lambda_v, \lambda_\Xi) }{ \frac{1}{2}|\lambda_v|^2 + G(\lambda_\Xi) } = 0 
\end{aligned}
\end{equation}
where  $\lambda_v \in \R^3$, $\lambda_\Xi = (\lambda_F, \lambda_Z, \lambda_w) \in 
\mat^{3\times 3}\times \mat^{3\times 3}\times\R=\R^{19}$.
The Young measure is connected with the map $y$ through the requirements that
(almost everywhere) 
\begin{equation}
\label{defavg}
F = \langle \nu, \lambda_F \rangle  \, , \quad  v = \langle \nu , \lambda_v \rangle \, , \quad
\Xi = \langle \nu, \lambda_Xi \rangle \,.
\end{equation}
The pair $(y,\nu)$ is a \emph{measure-valued solution} to ~\eqref{mainI} if for $i=1, 2 ,  3$
\begin{align}
\partial_t v_i - \partial_\alpha 
\big\langle \nu, \frac{\partial G}{\partial\Xi^A}  ( \lambda_\Xi ) \, \frac{\partial\Phi^A}{\partial F_{i\alpha}} (\lambda_F) \big\rangle = 0
\label{defmv1}
\\
\intertext{and  for $A=1,\dots , 19$}
\partial_t \Phi^A (F)   - 
\partial_\alpha  \big ( \frac{\partial\Phi^A}{\partial F_{i\alpha}}(F)  v_i  \big ) = 0
\label{defmv2}
\end{align}
in distributions with
\begin{equation}
\label{defmv3}
\Xi  = \Phi ( \langle \nu, \lambda_F \rangle ) = \Phi(F) \, .
\end{equation}
The solution is said to be a {\em dissipative measure-valued solution
with concentration} if it is a measure-valued solution which verifies 
in addition:
\[
\iint\,\frac{d \theta}{d t} \,\Bigl(\langle\nu,\eta\rangle
+\gamma\Bigr)\ \dxdt +\int\theta(0)\eta_{0}(x)dx\geq 0\,,
\]
for all non-negative functions $\theta=\theta(t) \in \mathrm{C}^1_c ([0,T))$ with $\theta \ge 0$. Here
$\eta_0$ means the entropy $\eta$ evaluated on the initial data and
$\gamma$ is the non-negative concentration measure.
\end{definition}
Under the following additional growth assumptions on the function $G$:
\begin{itemize}
\item[(H1)] $G \in C^3(\mat^{3\times 3} \times \mat^{3\times 3} \times \R ; [0,\infty))$ is a strictly convex 
function satisfying for some $\gamma >0$ the bound $D^2 G \ge \gamma > 0$,

\item[(H2)] $G(F,Z,w) \ge c_1 ( |F|^p +  |Z|^q +  |w|^r + 1)  - c_2$ where $p\in (4, \infty), \ \ q, r \in [2,\infty)$,

\item[(H3)] $G(F,Z,w) \le c (  |F|^p +  |Z|^q + |w|^r +1)$,

\item[(H4)$^\prime$] $| \partial_F G| + |\partial_Z G|^{\frac{p}{p-1}} + |\partial_w G|^{\frac{p}{p-2} } \le o(1) (  |F|^p +  |Z|^q + |w|^r +1)$ \quad where $o(1) \to 0$ as $|\Xi| \to \infty$,
\end{itemize}

 existence of dissipative mv solutions as well as a weak-strong uniqueness result are proven, cf \cite{DeStTz2} and \cite{DeStTz}.
\begin{theorem}
Let $G$ satisfy $(H1)-(H3)$, $(H4)'$ and let  $(y, \nu, \gamma)$ be a dissipative measure-valued 
solution in the sense of definition \ref{defmvpcel}.
If the initial data equal those of a Lipschitz bounded solution $(\hat{v},\hat{F})\in W^{1,\infty}(\overline{Q}_T)$: 
$$(v(x,0),\Xi(x,0))=(\hat v(x,0),\Phi(\hat F(x,0)))$$
then $\gamma$ is zero, $(v,\Xi)=(\hat v,\Phi(\hat F))$ 
and $\nu=\delta_{\hat v,\Phi(\hat F)}$.
\end{theorem}

%%%%%%%%%%%
%%%%%%%%%%
\section{Weak-strong uniqueness for general hyperbolic conservation laws}\label{sec5}
In \cite{CleoTz} the method of relative entropy is extended to a more general class of problems of hyperbolic and hyperbolic-parabolic type. In this section we describe the results of that paper.
%Consider a system of hyperbolic-parabolic conservation laws
%\begin{equation}
%\label{intro-hyppar}
%\begin{aligned}
%\partial_t A(u) + \partial_\alpha F_\alpha (u) = \eps \partial_\alpha ( B_{\alpha \beta} (u) \partial_\beta u)\;,
%\end{aligned}
%\end{equation}
%where $u(t,x)$ takes values in $\R^n$, $t \in \R^+$, $x \in \R^d$ and $A , F_\alpha : \R^n \to \R^n $ , $B_{\alpha \beta} : \R^n \to \R^{n %\times n}$ are given smooth functions with
%$\alpha, \beta = 1, ..., d$. It is assumed that the associated hyperbolic problem
Consider the following hyperbolic problem
\begin{equation}
\label{intro-hypcl}
\begin{aligned}
\partial_t A(u) + \partial_\alpha F_\alpha (u) = 0,
\end{aligned}
\end{equation}
where $u(t,x)$ takes values in $\R^n$, $t \in \R^+$, $x \in \R^d$ and $A , F_\alpha : \R^n \to \R^n $ are given smooth functions with $\alpha = 1,\cdots, d$. 
%, $B_{\alpha \beta} : \R^n \to \R^{n \times n}$ are given smooth functions with $\alpha, \beta = 1, ..., d$.
It is assumed that this system is symmetrizable. %in the sense of  Friedrichs and Lax \cite{FL71}.
The following hypotheses are assumed throughout.
\begin{enumerate}
 \item[($\text{H}_1$)] \emph{ $A:\R^n\to\R^n$ is a $C^2$ globally invertible map,}
 \item[($\text{H}_2$)] \emph{  existence of an entropy-entropy flux pair $(\eta,q)$, that is $\exists$ $G:\R^n\to\R^n$, $G=G(u)$ smooth such that
$$
\begin{aligned}
\nabla \eta &= G \cdot \nabla A
\\
\nabla q_\alpha &= G \cdot \nabla F_\alpha \, , \quad  \alpha = 1, ..., d \, ,
\end{aligned}
$$}
 \item[($\text{H}_3$)] \emph{ the symmetric matrix $\nabla^2 \eta (u) - G(u) \cdot  \nabla^2 A (u)$ is positive definite,}
\end{enumerate}
together with the following growth assumptions on the entropy $\eta(u)$, the functions $F_\alpha$ and $A$
\begin{equation}\label{hypetap}
\beta_1 (|u|^p+1)-B\le \eta(u)\le \beta_2(|u|^p+1) \qquad \text{for } u\in\mathbb{R}^n
\tag{A$_1$}
\end{equation}
for some positive constants $\beta_1$, $\beta_2$, $B$  and for some $p \in (1,\infty)$.
\begin{equation}\label{hypfgrowth}
\frac{|F_\alpha(u)|}{\eta(u)}=o(1)\qquad \text{as } |u|\to\infty,\quad\alpha=1,\dots,d \, ,
\tag{A$_2$}
\end{equation}
\begin{equation}\label{hypAgrowth}
\frac{|A(u)|}{\eta(u)}=o(1)\qquad \text{as } |u|\to\infty \, .
\tag{A$_3$}
\end{equation}
\begin{equation}\label{S2.2 A4}
\frac{1}{C}( |A(u)|^q + 1) - B \le \eta(u)\le C( |A(u)|^q + 1) ,\qquad q>1,
\tag{A$_4$}
\end{equation}
for some uniform constant $C>0$ and $B > 0$.%, which amounts  to control on the growth of $v = A(u)$.

\begin{definition}\label{defdissi}
A dissipative measure-valued solution $(u,\boldsymbol{\nu},\boldsymbol{\gamma})$ with concentration to~\eqref{intro-hypcl} consists of $u\in L^\infty(L^p)$, a Young measure $\boldsymbol{\nu}=({\boldsymbol{\nu}}_{(x,t)})_{\{(x,t)\in \bar{Q}_T\}}$ and a non-negative Radon measure $\boldsymbol{\gamma}\in\mathcal{M}^+(Q_T)$ such that
$u (x,t)  = \langle {\boldsymbol{\nu}}_{(x,t)} , \lambda \rangle$ and
\begin{equation}\label{dfmv1}
\iint\langle {\boldsymbol{\nu}}_{(x,t)}, A_i(\lambda)\rangle \partial_t\varphi_{i}\,dx\,dt+\iint \langle {\boldsymbol{\nu}}_{(x,t)}, F_{i,\alpha}(\lambda)\rangle \partial_\alpha\varphi_{i}dx\,dt+\int \langle\boldsymbol{\nu}_{(x,0)},A_i\rangle\varphi_i(x,0)dx=0
\end{equation}
for $i=1,\dots,n$ and any $\varphi\in \mathrm{C}^1_c(Q\times[0,T))$, and
\begin{equation}\label{dfmv2}
\iint\frac{d\xi}{dt} \left[\langle {\boldsymbol{\nu}}_{(x,t)}, \eta(\lambda)\rangle dxdt+\boldsymbol{\gamma}(dxdt)    \right]+\int \xi(0)\left[\langle {\boldsymbol{\nu}}_{(x,0)}, \eta\rangle dx+\boldsymbol{\gamma}_0(dx)  \right]\ge 0,
\end{equation}
 for all $\xi=\xi(t)\in \mathrm{C}^1_c([0,T))$ with $\xi\ge 0$.
\end{definition}
The following theorem provides weak measure-valued versus strong uniqueness in the $L^p$ framework for $1<p<\infty$.
\begin{theorem}
\label{thmweakstrong}
Suppose that $(H1)-(H3)$ hold, the growth properties~\eqref{hypetap}--\eqref{S2.2 A4} are satisfied,  and the entropy  $\eta(u) \ge 0$. 
Assume that $(u,\boldsymbol{\nu}, \boldsymbol{\gamma})$ is a dissipative measure-valued solution,  $u = \langle \boldsymbol{\nu}_{(x,t)} , \lambda \rangle$,  and  $\bar{u}\in W^{1,\infty}(\overline{Q_T})$ 
is a strong solution to~\eqref{intro-hypcl}. Then, if the initial data 
satisfy $\boldsymbol{\gamma}_0=0$ and ${\boldsymbol{\nu}_0}_x=\delta_{\bar{u}_0}(x)$, then $\boldsymbol{\nu}=\delta_{\bar{u}}$ and $u=\bar{u}$ almost everywhere on $Q_T$.
\end{theorem}
\begin{remark}
The work of Christoforou and Tzavaras \cite{CleoTz} is a generalization of the results presented in \cite{BrDLSz}, where the case $A(u)=u$ was considered and concentration effects were ignored.
\end{remark}

%%%%%%%%%%%
%%%%%%%%%%%
\section{General relative entropy method in mathematical biology}\label{sec6}
In this section we give a short overview of an extension of relative entropy method to linear PDEs which was introduced in the context of biological systems in \cite{MiPeRy} and further developed in \cite{MiMiPe} and \cite{MiMiPe2}. It was extended to measure solutions with measure initial data in \cite{GwiWie}. Here the relative entropies are a family of renormalizations to the initial linear problem, obtained by multiplying the original equation by a nonlinear function.\\
The notation used in this section is intentionally inconsistent with the remainder of this paper -- however it is consistent with the notation used in the above mentioned papers on the subject.\\ 
Following \cite{Pert} and \cite{GwiWie} we consider the McKendrick-Von Foerster equation 
\begin{equation}\label{originalrenewal}
\begin{aligned}
\partial_t n(t,x)+\partial_x n(t,x)&=0\hspace{0.3cm}\text{on $(\Rbb^+)^2$},\\
n(t,x=0)&=\int_0^\infty B(y) n(t,y)dy,\\
n(x,t=0)&=n^0(x).
\end{aligned}
\end{equation}
Here, $n(x,t)$ denotes the population density at time $t$ with age $x$, and $B\in L^\infty(\Rbb^+;\Rbb^+)$ is a birth rate such that  
\begin{equation*}
\int_0^\infty B(x)dx>1.
\end{equation*}
The associated primal and dual eigenvalue problems have the form
%Under these assumptions, it can be shown that there exist uniquely determined solutions of the primal and dual eigenvalue problems,
\begin{equation}\label{primal}
\begin{aligned}
\partial_x N(x)+\lambda_0 N(x)&=0,\hspace{0.2cm}x\geq0,\\
N(0)&=\int_0^\infty B(y)N(y)dy,\\
N>0,&\hspace{0.3cm}\int_0^\infty N(x)dx=1
\end{aligned}
\end{equation}
and
\begin{equation*}
\begin{aligned}
-\partial_x\phi(x)+\lambda_0\phi(x)=\phi(0)B(x),\hspace{0.2cm}x\geq0,\\
\phi\geq0,&\hspace{0.3cm}\int_0^\infty N(x)\phi(x)dx=1,
\end{aligned}
\end{equation*}
where $\lambda_0>0$.
Under the above assumptions both these problems posses a unique solution.
In fact, it can be seen that the solution of \eqref{primal} is given by $N(x)=\lambda_0 e^{-\lambda_0x}$. Since the death rate is ignored and the birth rate integrates to more than one, an exponential growth of the population is expected. In order to quotient out this growth, we set 
\begin{equation*} 
\n(t,x)=n(t,x)e^{-\lambda_0t}.
\end{equation*}
Then \eqref{originalrenewal} becomes
\begin{equation}\label{renewal}
\begin{aligned}
\partial_t\n(t,x)+\partial_x\n(t,x)+\lambda_0 \n(t,x)&=0\hspace{0.3cm}\text{on $(\Rbb^+)^2$},\\
\n(t,x=0)&=\int_0^\infty B(y)\n(t,y)dy,\\
\n(x, t=0)&=n^0(x).
\end{aligned}
\end{equation}
It can be shown that ~\eqref{renewal} has a unique solution in the weak sense for any $n^0\in\mathcal{M}^+([0,\infty))$, see \cite{GLMC}. The following results on long-time asymptotics of this solution are proven in \cite{GwiWie}.
\begin{theorem}\label{expasympthm}
Let $n^0\in\mathcal{M}^+([0;\infty))$. Then there is $y_0>0$, $\sigma>0$ and a bounded function $\eta$, positive on $\supp\phi$, such that the solution of the renewal equation satisfies
\begin{equation}\label{expasymp}
\int_0^\infty \eta(x)d|\n(t,x)-m_0N(x)dx|\leq e^{-\sigma(t-y_0)} \int_0^\infty \eta(x)d|\n^0(x)-m_0N(x)dx|,
\end{equation}
where $m_0=\int_0^\infty\phi(x)\mathrm{d}n^0(x)$.
\end{theorem}
\begin{theorem}\label{alternative}
Assume in addition there exists $C>0$ such that $B(x)\geq C\phi(x)$.
Let $n^0\in\mathcal{M}^+([0;\infty))$. Then the solution of the renewal equation satisfies
\begin{equation}\label{asymp}
\lim_{t\to\infty}\int_0^\infty \phi(x)d|\n(t,x)-m_0N(x)dx|=0,
\end{equation}
where $m_0=\int_0^\infty\phi(x)dn^0(x)$.
\end{theorem}

\section{Acknowledgement}
This work was partially supported by the Simons -- Foundation grant 346300 and the Polish Government MNiSW 2015-2019 matching fund.
T.D acknowledges the support of the National Science Centre, DEC-2012/05/E/ST1/02218. The research was partially supported by the Warsaw
Center of Mathematics and Computer Science. P.G and A.\'S-G received support from the National Science Centre (Poland), 2015/18/M/ST1/00075.

%%%%%%%%%%%%%%%%%%%%%%%%%%%%%%%%%%%%%%%%%%%%%%%%%%%%%%%%%%%%%%%%%%%%%%%%%%%%%%%
\def\cprime{$'$} \def\ocirc#1{\ifmmode\setbox0=\hbox{$#1$}\dimen0=\ht0
  \advance\dimen0 by1pt\rlap{\hbox to\wd0{\hss\raise\dimen0
  \hbox{\hskip.2em$\scriptscriptstyle\circ$}\hss}}#1\else {\accent"17 #1}\fi}


\begin{thebibliography}{10}
\bibitem{AlBo}
J.-J. Alibert and G. Bouchitt\'e.
\newblock Non-uniform integrability and generalized young measure.
\newblock {\em J. Convex Anal.}, {\bf 4}:129--148, 1997.

\bibitem{BeFeNo}
P.~Bella, E.~Feireisl, A.~Novotn\'{y}.
\newblock Dimension reduction for compressible viscous fluids.
\newblock {\em Acta Appl. Math.} {\bf 134}:111--121, 2014.

\bibitem{BrDLSz}
Y.~Brenier, C.~De Lellis and L.~Sz\'ekelyhidi Jr.
\newblock Weak-strong uniqueness for measure-valued solutions.
\newblock {\em Comm. Math. Phys.}, {\bf 305}(2):351--361, 2011.

\bibitem{BrzFei}
J.~B\v{r}ezina, E.~Feireisl.
\newblock Measure-valued solutions to the complete Euler system.
\newblock arXiv:1702.04870v1.

\bibitem{BuGG}
M.~Bul\'{i}\v{c}ek, P.~Gwiazda, A.~\'Swierczewska-Gwiazda.
\newblock On unified theory for scalar conservation laws with fluxes and sources discontinuous with respect to the unknown.
\newblock {\em J. Diff. Eq.}, {\bf 262}:313--364, 2017.

\bibitem{BuGG2}
M.~Bul\'{i}\v{c}ek, P.~Gwiazda, A.~\'Swierczewska-Gwiazda.
\newblock Multi-dimensional scalar conservation laws with fluxes discontinuous in the unknown and the spatial variable.
\newblock {\em Math. Models Methods Appl. Sci.}, {\bf 23}(3):407--439, 2013.

\bibitem{BuGG3}
M.~Bul\'{i}\v{c}ek, P.~Gwiazda, J.~Malek, A.~\'Swierczewska-Gwiazda.
\newblock On scalar hyperbolic laws with discontinuous flux.
\newblock {\em Math. Models Methods Appl. Sci.}, {\bf 21}(1), 2011.

\bibitem{Chen}
G-Q.~Chen, N.~Even, M.~Klingenberg.
\newblock Hyperbolic conservation laws with discontinuous fluxes and hydrodynamic limit for particle systems.
\newblock {\em J. Diff. Eq}, {\bf 245}:3095-3126, 2008.

\bibitem{CleoTz}
C.~Christoforou, A.E.~Tzavaras.
\newblock Relative entropy for hyperbolic-parabolic systems and application to the constitutive theor of thermoviscoelasticity.
\newblock arXiv:1603.08176.

\bibitem{Daf}
C.~Dafermos.
\newblock The second law of thermodynamics and stability.
\newblock {\em Arch. Rational Mech. Anal.} {\bf 70}:167--179, 1979.

\bibitem{Daf2}
C.~Dafermos.
\newblock {\em Hyperbolic conservation laws in continuum physics. Third edition}.
\newblock Springer-Verlag, Berlin, 2010.

\bibitem{DLSz}
C.~De~Lellis, L.~Sz\'{e}kelyhidi Jr.
\newblock On admissibility criteria for weak solutions of the Euler equations.
\newblock {\em Arch. Rational Mech. Anal.} {\bf 195}(1):225--260, 2010.

\bibitem{DeStTz}
S.~Demoulini, D.M.A.~Stuart and A.E.~Tzavaras.
\newblock Weak-strong uniqueness of dissipative measure-valued solutions for polyconvex elastodynamics.
\newblock {\em Arch. Ration. Mech. Anal.}, {\bf 205}(3):927--961, 2012.

\bibitem{DeStTz2}
S.~Demoulini, D.M.A.~Stuart, A.E.~Tzavaras, 
\newblock A variational approximation scheme for three-dimensional elastodynamics with polyconvex energy.
\newblock {\em Arch. Rational Mech. Analysis} {\bf 157}:325--344, 2001.

\bibitem{DiP}
R.~J. DiPerna.
\newblock Measure-valued solutions to conservation laws.
\newblock {\em Arch. Ration. Mech. Anal.}, {\bf 88}(3):223--270, 1985.

\bibitem{DiMa}
R.~J. DiPerna and A.~J. Majda.
\newblock Oscillations and concentrations in weak solutions of the incompressible fluid equations.
\newblock {\em Comm. Math. Phys.}, {\bf 108}(4):667--689, 1987.

\bibitem{FeJiNo}
E.~Feireisl, B.J.~Jin and A.~Novotn{\' y}.
\newblock Relative entropies, suitable weak solutions, and weak-strong uniqueness for the compressible {N}avier-{S}tokes system.
\newblock {\em J. Math. Fluid Mech.}, {\bf 14}:712--730, 2012.

\bibitem{FeGSG}
E.~Feireisl, P.~Gwiazda, A.~\'Swierczewska-Gwiazda and E.~Wiedemann.
\newblock Dissipative measure-valued solutions to the compressible Navier-Stokes system.
\newblock {\em Calc. Var. PDE}, {\bf 55}(6):Art. 141, 20 pp., 2016.

\bibitem{FeNo}
E.~Feireisl, A.~Novotn\'{y}.
\newblock Weak-strong uniqueness property for the full Navier-Stokes-Fourier system.
\newblock {\em Arch. Rational Mech. Anal.} {\bf 204}:683--706, 2012.


\bibitem{FjKaMiTa}
U.S.~Fjordholm, R.~K\"{a}ppeli, S.~Mishra, E.~Tadmor.
\newblock Construction of approximate entropy measure-valued solutions for hyperbolic systems of conservation laws.
\newblock {\em Found. Comput. Math.} {\bf 17}(3):763--827, 2017.

%\bibitem{Ger}
%P.~Germain.
%\newblock Weak-strong uniqueness for the isentropic compressible
%  {N}avier-{S}tokes system.
%\newblock {\em J. Math. Fluid Mech.}, {\bf 13}(1):137--146, 2011.

\bibitem{FjLaMi}
U.S.~Fjordholm, S.~Lanthaler, S.~Mishra.
\newblock Statistical solutions of hyperbolic conserwation laws I: Foundations.
\newblock arXiv:1605.05960.

\bibitem{FjMiTa}
U. S. Fjordholm, S. Mishra and E. Tadmor.
\newblock On the computation of measure-valued solutions.
\newblock {\em Acta Numerica}, {\bf 25}:567-679, 2016.

\bibitem{GiTz}
J.~Giesselmann and A.E.~Tzavaras.
\newblock Stability properties of the Euler-Korteweg system with monotone pressures.
\newblock {\em Applicable Analysis}, 2017.

\bibitem{Gwi}
P.~Gwiazda.
\newblock On measure-valued solutions to a two-dimensional gravity-driven avalanche flow model.
\newblock {\em Math. Methods Appl. Sci.}, {\bf 28}(18):2201--2223, 2005.

\bibitem{GLMC}
P.~Gwiazda, T.~Lorenz, A.~Marciniak-Czochra.
\newblock A nonlinear structured population model: Lipschitz continuity of measure valued solutions with respect to model ingredients.
\newblock {\em J. Diff. Eq.} {\bf 248}:2703--2735, 2010. 

\bibitem{GwiWie}
P.~Gwiazda, E.~Wiedemann
\newblock Generalized entropy method for the renewal equation with measure data.
\newblock {\em Commun. Math. Sci.} {\bf 15}(2):577--586, 2017.

\bibitem{GwSwWi}
P.~Gwiazda, A.~\'Swierczewska-Gwiazda and E.~Wiedemann.
\newblock Weak-strong uniqueness for measure-valued solutions of some compressible fluid models.
\newblock {\em Nonlinearity}, {\bf 28}(11):3873--3890, 2015.

\bibitem{GwSwWitZim}
P.~Gwiazda, A.~\'Swierczewska-Gwiazda, P.~Wittbold, A.~Zimmermann.
\newblock Multi-dimensional scalar balance laws with discontinuous flux.
\newblock {\em J. Funct. Anal.}, {\bf 267}(8):2846--2883, 2014.

\bibitem{KroZaj}
D.~Kr\"{o}ner, W.~Zaj\k{a}czkowski.
\newblock Measure-valued solutions of the Euler system for ideal compressible polytropic fluids.
\newblock {\em Math. Methods Appl. Sci.} {\bf 19}(3):235--252, 1996.

\bibitem{Kru}
S.N.~Kru\v{z}kov.
\newblock First order quasilinear equations in several independent variables.
\newblock {\em Math. USSR-Sb} {\bf 10}:217--243, 1970.

\bibitem{LiPeTa}
P.L.~Lions, B.~Perthame, E.~Tadmor.
\newblock A Kinetic formulation of scalar multidimensional conservation laws.
\newblock {\em J. AMS} {\bf 7}:169--191, 1994.

\bibitem{MiMiPe}
P.~Michel, P.~Mischler and B.~Perthame.
\newblock General entropy equations for structured population models and scattering.
\newblock {\em C. R. Acad. Sci. Paris}, Ser. I 338: 697--702, 2004.

\bibitem{MiMiPe2}
P.~Michel, P.~Mischler and B.~Perthame.
\newblock General relative entropy inequality: an illustration on growth models.
\newblock {\em J. Math. Pures Appl.} {\bf 84}: 1235--1260, 2005.

\bibitem{MiPeRy}
S.~Mischler, B.~Perthame, L.~Ryzhik.
\newblock Stability in a Nonlinear Population Maturation Model.
\newblock {\em Math. Models Methods Appl. Sci.} {\bf 12}:1751--1772, 2002.
lin, Lecture  Notes in Math. 1048: 60--110.

\bibitem{rokyta}
J. Necas, J. Malek, M. Rokyta, M. Ruzicka.
\newblock Weak and Measure-Valued Solutions to Evolutionary PDEs. Chapman and Hall/CRC, 1996.

\bibitem{Neustup}
J.~Neustupa.
\newblock Measure-valued solutions of the {E}uler and {N}avier-{S}tokes
  equations for compressible barotropic fluids.
\newblock {\em Math. Nachr.}, {\bf 163}:217--227, 1993.

\bibitem{Pert}
B.~Perthame.
\newblock {\em Transport Equations in Biology}.
\newblock Frontiers in Mathematics, Birkh\"{a}user Basel, 2007.

\bibitem{PeTa}
B.~Perthame, E.~Tadmor.
\newblock A kinetic equation with kinetic entropy functions for scalar conservation laws. \newblock {\em Commun. Math. Phys.} {\bf 136}:501--517, 1991.

\bibitem{PeTz}
B.~Perthame, A.E.~Tzavaras.
\newblock Kinetic formulation for systems of two conservation laws and elastodynamics.
\newblock {\em Arch. Ration. Mech. Anal.} {\bf 155}(1):1--48, 2000.

\bibitem{SzWi}
L.~Sz\'{e}kelyhidi Jr., E.~Wiedemann.
\newblock Young measures generated by ideal incompressible fluid flows.
\newblock {\em Arch. Rational Mech. Anal.} {\bf 206}(1):333--366, 2012.

\bibitem{Szepessy}
A.~Szepessy
\newblock An existence result for scalar conservation laws using measure valued solutions.
\newblock {\em Comm. PDE}, {\bf 14}(10):1329--1350, 1989.

\bibitem{Tar}
L.~Tartar.
\newblock Compensated compactness method applied to systems of conservation laws.
\newblock {\em Systems Nonlinear PDE}, Ball, J.M (ed.) NATO ASI Series, C.Reidel Publishing Col., 1983.

\bibitem{Wiedemann}
E.~Wiedemann.
\newblock Weak-strong uniqueness in fluid dynamics.
\newblock arXiv:1705.04220.

\end{thebibliography}
\end{document}